\documentclass[12pt]{amsart}
\usepackage{graphicx}
\usepackage{amssymb}
\usepackage{setspace}
\usepackage{amsmath}     
\usepackage{stackrel}    
\usepackage{stackengine} 

\usepackage{tcolorbox}   

\input xypic \xyoption{curve}

\usepackage{appendix}

\newcommand{\B}[1]{{\mathbf#1}} 
\newcommand{\C}[1]{{\mathcal#1}} 

\theoremstyle{plain}

\newtheorem{prop}{Proposition}[section]

\theoremstyle{definition}
\newtheorem{defin}{Definition}[section]

\theoremstyle{definition}
\newtheorem{example}{Example}[section]

\theoremstyle{remark}
\newtheorem{rem}{Remark}[section]

\usepackage{fancyhdr}
\begin{document}
\title{On Hodge Structures and Periods - Part I: Algebraic Aspects}
\author{Lucian M. Ionescu}
\address{Department of Mathematics, Illinois State University, IL 61790-4520}
\email{lmiones@ilstu.edu}
\date{Sept. 2020} 
\keywords{Hodge structure, Periods, Motives}

\begin{abstract}
We review Hodge structures, relating filtrations, Galois Theory and Jordan-Holder structures.

The prototypical case of periods of Riemann surfaces is compared with the Galois-Artin framework of algebraic numbers.
\end{abstract}

\maketitle
\setcounter{tocdepth}{3} 
\tableofcontents

\section{Introduction}\label{S:Intro}
Periods \cite{KZ} are also ``numbers'' computed as Feynman integrals accurately representing scattering amplitudes measured in particle accelerator experiments \cite{QM:Periods}.
To understanding their significance from a mathematical point off view, a comparison to algebraic numbers, also periods, but much better understood, is in order.
Yet the key structure involved is that of Hodge structure, which is reviewed from several directions: Hodge decomposition, filtration and representations.

While first sections are elementary and expository, the last part targets nc-Hodge structures in a speculative way. 
\cite{Simpson:HFnac, Corlette} ``allows for non-linearity at the substrate of Hodge structure'' \cite{KKP}, p.6, following the lessons from non-abelian cohomology and aiming to study the representations of the fundamental group, while \cite{KKP} generalizes the local systems idea, and relying on Rees construction to relate classical (abelian) and nc-Hodge structures in a functorial way,
compatible with the theory of variations.

Yet studying separately the topological structures (down-stairs, on the base) and algebraic structures (upstairs, on the fiber), {\em first}, before studying their interaction in “Differential/Algebraic Geometry”, perhaps clarifies some options when designing the theory.

Since classical Hodge structures, as representations of $C^*$, extend the complex structure given as a representation of $SO(2)$, the  author proposes a direct extension of such a ``loxodromic'' mix of rotations with a radial grading, to actions of $SL(2,C)$:
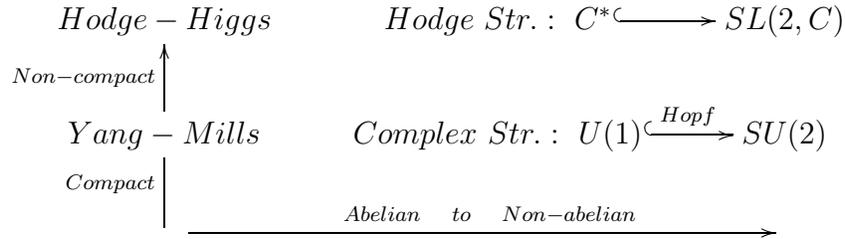
$$\xymatrix{ 
Hodge-Higgs          & Hodge\ Str.: \ C^*\ar@{^(->}[r]^{ } & SL(2,C)\\
\ar@{->}[u]^{Non-compact} Yang-Mills          & Complex\ Str.: \ U(1) \ar@{^(->}[r]^{\quad \quad \quad Hopf} & SU(2)\\
 \ar@{-}[u]^{Compact}\quad \ar[rr]^{\quad Abelian\quad to\quad Non-abelian} & &
 }$$
The above arrows relate to historical developments around Hodge Theory \cite{Wiki:NAHC} (see also the diagram from \S \ref{S:Conclusions}).

One may interpret this generalization as passing from genus zero mapping class group $C^*=MPG(S^2)$ to genus one $SL(2,C)=MCG(S^1\times S^1$), i.e. corresponding to the compactification of $R^{(2,0)}$ or the compactification of the conformal plane $R^{1,2}$ \cite{Conf-Geom}
\footnote{Unfortunately in low dimensions there are plenty of accidental isomorphisms, to rely on such a ``coincidence''; comments from the reader are welcome.}.

\section{Pure Hodge Structures}
Hodge decomposition, leading to Hodge structure and beyond to motives, originated from a duality first noted in Electromagnetism, formulated by Grassmann via complements, and later by Minkowski in his relativistic formulation of EM.

Subsequently it was set on firm foundations by Hodge, building on the work by de Rham in the context of complex manifolds.

As expected in Mathematics, an abstract algebraic structure was eventually extracted, called {\em Hodge structure}.
There are three main definitions, as starting points, with more or less direct connections\cite{Intro-HS}.

%
%
\subsection{Pure Hodge Structures}
Following \cite{Intro-HS}, \S2, we analyze the roles of various algebraic structures involved, with the following prospecting ideas in mind:
filtrations go together with Galois Theory and Jordan-Holder Theorem. while the complex conjugation as an involution is reminiscent of Cartan involutions (see Wiki), yielding various real forms of Lie algebras associated to one complex Lie algebra.

\subsubsection{Definitions}\label{HS-definitions}
First let us consider the three main definitions of a pure Hodge structure, and only then comment on the various relations and connections with other theories.

{\bf HS Definition I.}
A {\em pure Hodge structure of weight $n\in Z$}, denoted $(H_\B{Z}, \{H^{p,q}\}_{p+q=n}$, consists of a finitely generated free abelian group $H_\B{Z}$ (Hodge lattice, for short) together with a decomposition of its complexification, compatible with complex conjugation:
$$H_C:= H_Z \otimes_Z C = \oplus_{p+q=n} H^{p,q}, \quad \bar{H^{p,q}}=H^{q,p},$$.

\vspace{.1in}
Alternative to this definition based on an explicit ``block'' decomposition, generalizing the concept of grading, it is common to define a Hodge structure via a filtration.

{\bf HS Definition II}.
A {\em (pure) Hodge structure of weight $n\in \B{Z}$}, denoted $(H_\B{Z}, \{F^p\}$, consists in a lattice $H_\B{Z}$ (as above), together with a filtration of its complexification, with conjugation acting as a complement \footnote{Reminiscent of Grassmann complements and Hodge *-duality: \cite{Grassmann-complements}.}:
$$H_C=F^0\supset F^1 \cdots \supset F^n \supset 0,  \quad \bar{F^p} \oplus F^q\cong H_C.$$

Note that in the 2nd definition the isomorphism is not part of the data.

\vspace{.1in}
It is immediate that the two definitions are equivalent, for instance by considering the canonical filtration associated to the decomposition: 
$$F^p=H^{n,0}\oplus ... \oplus H^{p,n-p}, \ p=0..n.$$
Both structures will be referred to as Hodge filtration / Hodge decomposition of the corresponding Hodge structure.

%
\vspace{.1in}
Additional insight comes from considering the 3rd definition of a Hodge structure in terms of representation Theory.
This should not be very surprising if we recall that complex numbers, from a geometric point of view, are just a fundamental representation of the group of rotations $SO(2)$.
Below we will follow \cite{Geeman}, p.1.

But first, some preliminaries.

The multiplicative complex numbers are viewed as the group of rotations and homoteties
through the identification of the algebraic form $z=a+ib$ with the geometric form: $$\B{C}^*=\{s(a,b):=
\begin{pmatrix} a & b \\ -b & a  \end{pmatrix} \in GL(2, \B{R)}: a^2+b^2\ne 0\}, \quad z=a+ib\mapsto s(a,b).$$ 
Note that the eigenvalues of $s(a,b)$ are complex conjugate $z=a+ib$ and $\bar{z}=a-ib$, with eigenspaces in the complexification of the fundamental representation used above:
$$GL(2,\B{C})=GL(2,\B{R}) \otimes_\B{R} \B{C}.$$
When defining the Hodge structure below, not only the lattice structure is again omitted at this point (when comparing with Definition I and II), but the emphasis will be on the {\em real form} of Hodge structure;
the passage to the complexified (form of a real) Hodge structure, satisfying the usual 
compatibility condition with complex conjugation (Galois group) comes ``for free'' when extending coefficients from $\B{R}$ to $\B{C}$, via $\cdot \otimes_\B{R}\C{B}$.

\begin{defin}(\cite{Geeman}, p.1)
An {\em algebraic representation} of the group $C^*$ is a representation $h:C^*\to GL(V_\B{R})$ such that the matrix coefficients are real coefficients polynomials in 
$a,b$ and $(a^2+b^2)^{-1}$.

It is said to be {\em real of weight $k$} if its restriction to $R_+$ is $h(t)=t^k$. 
\end{defin}
\begin{rem}\label{R:complexification}
Complexifying the representation, i.e. following it by inclusion $GL(V_R)\to GL(V_C)$, corresponding to the complexification of the real vector space, yields a matrix representation with entries polynomials in $z, \bar{z}$ and $(z\bar{z})^{-1}$. 
\end{rem}

\begin{example}
The fundamental representation of $C^*$ on $R^2$ is algebraic of weight $k=1$.

In general, a real vector space comes with an obvious action of $R_+$, hence its complexification acquires a $C^*$ algebraic representation via the $SO(2)$ action from the complex structure $J^2=-Id$, or ``multiplication by $i$, corresponding to the $(v,w)\mapsto (-w, v)$, typical of most ``doubling constructions''. 
\end{example}

\vspace{.1in}
We are now in position to provide the main definition of a Hodge structure, as pertaining to the goals targeted in this article.

{\bf HS Definition III.} A {\em rational Hodge structure} of weight $k\in \B{Z}$ is a $\B{Q}$-vector space $V$ together with an {\em algebraic representation} $h:C^* \to GL(V_\B{R})$, of the multiplicative group $C^*$.

\vspace{.1in}
\begin{rem}
Here the emphasis is, as we will see, on the Hodge decomposition, not on the underlying lattice
$H_Z$, as in Definitions I and II.

Note also that a {\em real} vector space $V_\B{R}$ comes by definition with the $(R_+,\cdot)$ action on the abelian group ($\B{Z}$-module), hence a complexification $V_\B{R}\otimes \B{C}$ endows it with a canonical complex structure, or equivalently an $SO(2)$ action.

With hindsight, the rich structure of the complex numbers (complex, symplectic, Kahler, Hodge), as an extension of the reals, comes from several ``places'':
$$\B{Z}\to (Z[i],+) \to (\B{C}, +), \quad (\B{C}, cdot)=(R_+,\cdot) \times SO(2;\B{R}).$$
\end{rem}

The correspondence between Hodge decompositions (from Def. I or II) and representations (Def. III), perhaps not so ``well known'' as it should, can be found in \cite{Geeman}, 1.4 Prop., p.2.

\begin{prop}
Given a rational vector space $V_Q$, 
there is a bijection (morphism?) between {\em algebraic representations of weight} $k$:
$$h:\B{C}^* \to GL(V_R), h(t)=t^k, t\in R$$
and rational Hodge structures $\{V^{p,q}\}_{p+q=k}$ of weight $k$,
such that the $(p,q)$ Hodge decomposition summand, denoted $V_{p,q}$, is the eigenspace of the associated $C^*$-automorphism /character $\lambda^{p,q}(z)=z^p\bar{z}^q$:
$$V^{p,q}=\{v\in V_C : h(z)v=\lambda^{p,q}v\}.$$
\end{prop}
\begin{proof}
``$\Rightarrow$'' 
Given an algebraic representation $h:R_+\times SO(2)\to GL(V_R)$, consider its complexification, conform Remark \ref{R:complexification}.
$C^*$ abelian implies the operators may be simultaneously diagonalized, i.e. 
there is a basis $ \{ v_i\}$ of $V_C$ of simultaneous eigenvectors
$h(z)v_i=\lambda_i(z)v_i$, for some characters $\lambda_i$.
Since $\lambda_i$ are assumed to be of polynomial type in $z,\bar{z}$, they are in fact monomials of the form $\lambda_i(z)=z^p\bar{z}^q$, for some $p,q$ depending on $i$ (index was omitted for simplicity of notation).

Now define the subspaces
$$V^{p,q}=\{v\in V_C : h(z)v=z^p\bar{z}^q v, \ z\in C^*.$$
By the above argument they span $V_C$, and have trivial intersection.
That that real part of the representation is real $h(t)=t^k$, ensures that
the decomposition is ``homogeneous of order $k$: 
the only non-zero eigenspaces are those with $p+q=k$.

Regarding complex conjugation requirement, since $h$ is a real representation complexified, the eigenspaces come in complex conjugate pairs: $ \overline{V^{p,q}} = V^{q,p}$ (just a linear algebra version of ``real coefficients polynomials have complex conjugate pairs of roots'').

This common spectral decomposition is by definition the Hodge decomposition (of weight $k$), of a  corresponding unique Hodge structure (conform HS Definition I, \S \ref{HS-definitions}).
Indeed, the Hodge lattice here is naturally associated with the chosen basis $ H_Z=Z\{v_i\}$.

``$ \Leftarrow$'' Conversely, given a rational Hodge structure on $V_Q$:
$$ V_Q\otimes_Q C = \oplus_{p+q} V^{p,q}, \quad \overline{V^{p,q}}=V^{q,p},$$
meaning that conjugation acts as identity on $V_{p,p}$, if it happens that $k=2p$ is even, and is an isomorphism:
$$ \bar{\ }:V^{p,q}\to V^{q,p}, p>q, \quad v\otimes1+w\otimes i\mapsto v\otimes1-w\otimes i.$$
In other words, for $p>q$ each $V^{p,q}=Re(p,q)\oplus i Im(p,q)$ decomposes into a direct sum of eigenspacees for conjugation, with corresponding eigenvalues $+1$ and $-1$ respectively,
with $Re(p,q)$ and $(Im(p,q)$ real vector subspaces of $V_R=V_Q\otimes R$.

Now we have the individual $C$-planes $V_p=Re(p,q)+i Im(p,q)$, for $p\ne q$, and a possibly additional ``purely real subspace'', invariant to conjugation $V^{p,p}$, if $k$ even.

With such an explicit structure, we can naturally define a $C^*$-``conformal action'' piece-wise on the ``upper-half'' blocks $p>q$:
$$h:C^*\to GL(V^{p,q}) : h(r e^{i\theta})=r^{2(p+q)} 
\begin{bmatrix}\cos(\theta) & \sin(\theta) \\ -\sin(\theta) & \cos(\theta)\end{bmatrix}^{p-q},$$
and mirror as conjugation on the lower blocks $V^{q,p}$, 
with a ``pure'' $r^{2k}$ action on the purely real subspace, if there is one.  \qed
\end{proof}
\begin{rem}
Note also that moreover, the above Hodge decomposition associated to such a real representation, is compatible with the Hodge lattice, i.e. it partitions the global basis into bases for the Hodge summands $V^{p,q}$.
\end{rem}

\begin{rem}
The parallel with Cartan method for classification of semisimple Lie algebras is poignant. 

The role of the Cartan subalgebra $h$ / maximal torus $T$ is played at the level of groups by $C^*$, maximal abelian in $SL(2,C)$, but not compact: 
$$ C^*=R_+ \times U(1) \to SL(2,C) \quad compare \ with \quad  U(1) \to SU(2,C),$$
allowing to simultaneously diagonalize the corresponding restriction to $h$ of the adjoint representation $ad_g$.
The characters $\lambda^{p,q}$ play the role of weights of a root system.
\end{rem}

\subsection{Miscellaneous considerations}
Of course the $SO(2)$ part of the representation $h$ on $C^*=R_+\times SO(2)$ is just the datum of a linear complex structure $J=h(i)$ on $V_R$. 

In conclusion, a pure Hodge structure on a vector space is a ``discrete variation'' of the usual complex structure, which is an extension of the real scalar multiplication to include rotations, to a group action of $C^*$ by ``conformal transformations'' with respect to an inner product (e.g. hermitean product, when considering a Kahler manifold), with a non-trivial action in the ``polar direction'': $h(N(z))v=N(z)^k v$.

\vspace{.1in}
When comparing the various definitions of Hodge structures, it is worth distinguishing the change of coefficients in two steps: from $Z->Z[i]$, i.e. the integral extension, and then $\cdot \otimes R$.

The appearance of such ``conformal representations'' is reminiscent of modular forms:
$$ f(\gamma(z))=(cz+d)k f(z), \ \gamma=\frac{az+b}{cz+d},$$
where the modular transformation $\gamma$ is an element of the mapping class group of the elliptic curves $SL(2,Z)$.

Here we are dealing with the maximal abelian subgroup (maximal torus) of (essentially) Mobius transformations $SL(2,C)$, i.e. its Cartan subalgebra.

How such algebraic representations are involved, as Hodge structures, is not apparent to the present author.

A ``mnemonic'' of the hierarchy of structures involved, is the following progression of ``enriching'' an abelian group with an action (module structure):
$$ Abelian\ group  \quad Real \ structure \quad Complex \ structure  \quad Hodge \ structure$$
\begin{equation}\label{D:structures}
\xymatrix @C=1pc{
Z \quad \ar@{^(->}[rr] & & \quad R_+ \quad \ar@{^(->}[rr] & & \quad C^*=R_+ \times SO(2) \quad \ar@{^(->}[rr] & & R_+ \times_k SO(2)
}\end{equation}

\subsection{Examples: Dolbeault Complexes}
In fact there is an underlying decomposition of the complexification of the exterior algebra
of differential forms, pertaining to the above algebraic considerations fiber-wise.
Since it is compatible with differentiation and {\em integrable} for complex manifolds, it leads to the Hodge decomposition of the Dolbeault complex.

Remaining as before at the linear algebra level \cite{RC-DiffGeom}), from the point of view of $C^*$-representations, the Dolbeault decomposition originates as follows.

It is enough to decompose differential 1-forms fiber-wise:
$$\Omega^1(X;C)=\Omega^{1,0}(X)+\Omega^{0,1}(X).$$
The corresponding spectral decomposition of $h_1\in GL(V_C)$, of weight $k=1$, is $z P_{E(i)}+\bar{z} P_{E(-i)}$, corresponding to the complexification.
The exterior algebra functor yields the higher weight decompositions $h_k=h_1^{\wedge^k}$.

\subsection{General Hodge Structures}
De Rham cohomology of a complex manifold decomposes in each degree $k$ as a Hodge structure of weight $k$; it is a direct sum of such Hodge structures.

Transferring to the linear algebra framework, we introduce the following equivalent formulation.
\begin{defin}
A {\em (general) rational Hodge structure} on a real vector space $V_R$ is a {\em polynomial $C^*$-action} $h:C^*\to Aut_R(V_C)$ on its complexification, compatible with complex conjugation
(Galois group equivariant).
\end{defin}
\begin{rem}
Again the ``lattice aspect is separated from this definition, as well as the complex conjugation aspects, in the process of changing coefficients from $Z$ or $Q$ to $R \to C$, which usually play a part together in the classical concept of Hodge structure.
\end{rem}
Such an action is a direct sum of Hodge structures of weights $k$, corresponding to the decomposition of a polynomial in $z$ and $\bar{z}$ into homogeneous components:
$$h(z,\bar{z})=\sum h_k(z,\bar{z})t^k., \quad h_k(z,\bar{z})=\sum_{p+q=k} z^p\bar{z}^q P_{V(p,q)},$$
as a spectral decomposition of $h_k$. 


\section{Polarization of Hodge Structures}
A polarization of a Hodge structure $(H_Z,V_Q)$\footnote{Here we will consider the lattice too.} plays both roles of metric and symplectic form, like a super-symmetric Hermitean structure: instead of $<,>=g+i\omega$, as for a Kahler manifold (or package, if considering just the linear algebra setup), $Q=g_{|H^+}\oplus \omega[1]_{|H^-}$, where $V_C:=V_Q\otimes_Q C=H^+\oplus H^-$ is the decomposition of a total Hodge structure in its even and odd weights parts. 

More precisely:
\begin{defin}
A {\em polarization on a Hodge structure $(H_Z, V_Q, h:C^*\to Aut_C(V_C))$} is a $Z$-bilinear form on $H_Z$ (Frobenius-like), such that its complexification satisfies the {\em Hodge-Riemann bilinear relations}:

(i) ``Orthogonality relations'': $Q$ is non-degenerate on $H^{p,q}\times H^{q,p}$ and zero otherwise;

(ii) ``Complex/Symplectic signature'' (Clifford?): $i^{p-q}Q>0$ on $V^{p,q}+V^{q,p}$.
\end{defin}
The Hodge decomposition may be presented in a slightly different way, in view of the Clifford algebra interpretation \cite{Geeman}, p.9.

With the polarized rational Hodge structure $(V_Z,Q,h)$ there is an associated Clifford algebra $Cl_Q(V_C)$, relative to the complexification of the lattice pairing $Q:V_Z\times V_Z\to Z$.
The Hodge decomposition into even and odd parts is the usual decomposition of the Clifford algebra as a $Z_2$-graded algebra:
$$Cl(Q)=C^+(Q)\oplus C^-(Q) =\oplus_k\oplus_{p+q=k} H^{p,q}.$$
In fact a Clifford algebra is also a Frobenius algebra \cite{Cliff-Frob}, and also a weak Hopf algebra \cite{Cliff-Hopf}, i.e. exhibiting a ``core algebraic structure'' suited for duality and much more (polarized Hodge structure).

In order to avoid general yet unsubstantiated considerations, we will exemplify the key connections between these algebraic structures in the case of Riemann surfaces.
This will explain, by example, the similarity with Cartan-Killing-Weyl presentation of semi-simple Lie algebras\footnote{The ``Lie algebra-to-bialgebra development of the theory leads to bialgebra quantization, Drinfeld doubles / Manin triples, other remarkable connections with the topics of Hodge structures and periods.}.

\subsection{The Elliptic Curves Example}
We will interpret the presentation of the moduli of polarized Hodge structures from \cite{Movasati:MHS}, 
with an emphasis on the representation theory approach to Hodge structures, combined with a root systems analogy when discussing the lattice data.

The theory of elliptic curves has analytic, topologic and algebraic aspects (Weierstrass, Riemann and Abel-Jacobi).

The Weierstrass family of elliptic curves $E(t_2,t_3): y^2=4x^3-t_2x+t_3$, with $(t_2,t_3)$ in the parameters domain $U_0$, corresponds on the algebraic side, via Weierstrass coordinates, to the torus $C/L$, the quotient via the lattice $L$ with periods $\omega=(\int_{\delta_1} dx/y, \int_{\delta_2} dx/y)$.
Here $\{\delta_i\}_{i=1,2}$ is a homological base for Betti homology $H^{Betti}_1(E(t),Z)$, with $t=(t_2,t_3)$, and $dx/y$ determines the basis of the de Rham cohomology $H_{dR}^1(E_t,C)$.

Now the natural Hodge decomposition $H^1_{dR}=H^{1,0}\oplus H^{0,1}$ is inherited from the Dolbeault complex.
Also the cup product provides a natural polarization on Betti cohomology lattice with dual basis $\delta_i^*$:
$$H^1(X_t,Z)\times H^1(E_t,Z)\to Z, \quad Q(a,b)=a\cup b,$$
$$\Rightarrow  \quad H^{Betti}_1(X,Z) \cong H^1(E_t,Z)^*\cong H^{Betti}_1(E_t,Z).$$
It is the integration pairing that yield the Betti- de Rham period isomorphism:
$$ P_{iso}:H^{Betti}_1(X,Z)^*\to H_{dR}^1(X,C).$$
and the non-trivial Hodge structure (decomposition) that encodes the periods of the algebraic torus,
from the algebraic side.
This is also defined as the period domain; when rotating in ``standard position'' with $\omega_1=1$ via an $SL(2,Z)$ transformation, the other period $\tau=\omega_2$ becomes the ``free'' affine parameter of the moduli space of the family of elliptic curves \cite{Movasati:MHS,Geeman}.

In conclusion, here we see the presence of a basis / dual basis (pairing, copairing), and the Hodge structure decomposition is determined by the ``gluing'' via an isomorphism of the {\em Frobenius form} type $A \to A^*$.
As mentioned before, there are various algebraic structures suited to be used for such a framework:
Clifford - Frobenius - Hopf algebras, with what we consider minor ``variations'' (conceptually), but yielding technically distinct details.

\subsection{From Killing-Cartan to Betti-de Rham-Hodge}
The theory of moduli of variations of polarized Hodge structures has several layers of complexity and ``incarnations'', author dependent.

The manifold theory of smooth, complex or projective theory involves, as already mentioned, the interplay between Topology, how the manifold is built, and the algebraic structures on fibers; it involves the so called {\em global period domain} (e.g. conform \cite{Intro-HS}), and representations of the fundamental group (additional comments will be given in the following sections).
Locally, on a trivialisation ``local model $\times$ fiber structure'', local period domain controls variations of the structure on the fibers, and it is a problem of integrability, e.g. from almost complex structures, structures defined on fibers, and complex structures, compatible with the local model.

It is the theory of linear Hodge structure, which is defined fiberwise, which can be comparred with the theory of seimi-simple Lie algebras (the fiber structure analog for groups), that helps clarify the main tools, concepts, tools and their roles in classifying polarized Hodge structures.
This analogy can also help generalize in a natural way Hodge structures, from Abelian to non-commutative.

A more detailed will appear elsewhere; for now, we will also present a dictionary:
$$\xymatrix{
Generators: & Killing-Cartan  & Polarized \ Hodge \ structures \\
& Root \ Systems & Lattice \ and\ Polarization \\
Variations: & Lie \ Theory & Deformation \ Theory
}$$
On the Lie algebra side, the choice of a Cartan subalgebra (maximal torus of ``periods'') determines the roots, as eigenvalue functionals; note that for semi-simple Lie algebras the inner product comes for ``free'', as the Killing form, and corresponds to the {\em real forms of the complexification}.

On the ``bialgebra side'' (we will eventually justify this elsewhere), the emphasis is on the representation theory (which in the Lie side is derived from the root system, Weyl group etc.), hence the lattice and the polarization are given first, as data encoding the ``hidden Frobenius duality'', and the simple roots geometric data (analog of Cartan matrix).
The Hodge structure, as a representation, is given from the beginning, encoding the hierarchy of extensions via filtrations or decompositions.
It is in some sense the analog of the Weyl action on the quiver of states, specifying the ``navigation'' via the creation and annihilation operators (the representation theory itself).

Thus the polarization is analogous to the Killing form associated to a real form of a Lie algebra, with associated Cartan involution and signature.
Note that the ``common denominator'' is the complexification, where the Hodge decomposition takes place.
The new aspect, pertaining to duality (``bialgebra quantization'') is the presence of its invariants, analogous to the invariants of an abelian group (discrete periods of the maximal torus).

\vspace{.1in}
This is a partial analogy, highlighting the differences also, allowing to complete and clarify the Hodge structure side accordingly.

It also justifies why a reasonable generalization of linear Hodge structure proceeds from $C^*$, the center of $SL(2,C)$ to this fundamental case of non-abelian and non-compact group.

Note that at the level of duality, we have:
$$\xymatrix{
         & \underline{Commutative} & \underline{Non-Commutative} \\
\underline{Compact} & Pontryagin & Tannaka-Krein \\ 
\underline{Non-compact}: & Hodge\ duality & nc.\ Hodge\ ...\ Mirror \ symmetry?
}
$$
\begin{rem}
The interplay between filtrations vs. grading, as a tower of extensions in the sense of Jordan-Holder Theorem and Hodge structures as representations, is quite instructive; alternatively it relates to the Rees construction / projectivity and Swan's Theorem, as mentioned later on in connection with the nc-Hodge structure from \cite{KKP}. 
\end{rem}

\section{Non-Commutative Hodge Structures}
Inspired by the nc-Hodge structure of \cite{KKP}, we will separate the Linear (fiber-wise) algebraic structure from Topological / integrability considerations,
and define a straight forward non-commutative almost Hodge structure, by extending the representation
of $C^*$ to algebraic representations of $SL(2,C)$ \ref{D:structures}, conform with the diagram from \S\ref{S:Intro}).

\begin{defin}
A {\em non-commutative Hodge structure} on a rational vector space  $V_Q$ is an algebraic representation $h:SL(2,C) \to GL(V_Q)$, such that its restriction to $C^*$ is a pure-rational Hodge structure.
\end{defin}
Polarizations of nc-Hodge structures will not be addressed at this stage, but a parallel with Killing forms of semis-simple Lie algebras, associated to Cartan involutions may constitute a guide in this direction.

Operations with such structures follow from linear algebra constructions \cite{Geeman}.

The lattice aspects of such linear nc-Hodge structures should be related to the modular group $SL(2,Z)$, in connection with the role of modular forms.

\section{Conclusions}\label{S:Conclusions}
Non-linear objects, like manifolds and associated functions, are studied via linear structures associated as fibers, e.g. tangent bundle etc.
The main idea is to capture the topology data used to glue and build the manifold, pairing Cech cohomology with the graded structures of the fibers, as in standard presentation of Hoge decomposition (coherrent sheaf cohomology)  $ H^k(X,C)=\sum_{p+q} H^p(X, \Omega^q)$.
Then various approaches, e.g. vector bundles, Higgs bundles, Yang-Mills Theory etc, generically called local systems, lead to representations of the fundamental group.

Some notable theorems provide ``bridging'' correspondences \cite{Wiki:NAHC}:
$$\xymatrix{
Stable \ Vector\ Bundles \dto^{K-H} \drto^{N-S} \rto^{S} & Higgs \ Bundles \dto^{C-S} & Modular\ Forms\\
Yang-Mills \ Theory \rto^{D} & \pi_1-representations\rto^{KKP} & nc-Hodge\ Theory
}
$$
The fact that non-abelain cohomology is well established in dimension 1, with its relation to the fundamental group (Hurewicz Theorem), was taken by Carlos Simpson as a ``motto'' for what nc-Hodge Theory is \cite{Simpson:HFnac}: ``{\em The study of nonabelian Hodge Theory may thus be thought of as the study of fundamental groups of algebraic varieties of compact Kahler maniffolds}''.

Returning to a more concrete level, we will end with a question for the reader.
If we just look for a ``linear non-commutative Hodge structure'', perhaps the extension from $Z[i]\to C^*$ to $SL(2;Z)\to SL(2,C)$ proposed above for a linear, pure, non-commutative Hodge structure, looks like a reasonable candidate; after all,
$SL(2,Z)$ is the analog of Mobius Transformations for the rational Riemann sphere
$P^1Z$ ... Then how do {\em modular forms} fit into this line of thought, at least in relation to the moduli space of elliptic curves, as their mapping class group?

\section{Acknowledgements}
The author thanks IHES for the wonderful research conditions provided during his recent visit, and for the helpful advice from Prof. Maxim Kontsevich; the loose ends and mistakes from this article are due to the author only, while trying to understand and study a new (for him) and exciting topic.


\end{document}